\documentclass[12pt]{amsart}
\usepackage[margin=1.25in]{geometry}
\usepackage{amsthm}
\usepackage{amsfonts}
\usepackage{hyperref}
\usepackage{graphicx}
\usepackage{placeins}
\graphicspath{ {./images/} }
\usepackage{amsmath}
\usepackage[utf8]{inputenc}
\newtheorem{proposition}{Proposition}[section]
\newtheorem{theorem}[proposition]{Theorem}
\newtheorem{lemma}[proposition]{Lemma}
\newtheorem{assumption}[proposition]{Assumption}
\newtheorem{remark}[proposition]{Remark}
\newtheorem{example}[proposition]{Example}
\newtheorem{question}[proposition]{Question}

\begin{document}

\title[Large Deviations for High Gaussian Minima]{Large Deviations for High Minima of Gaussian Processes with Nonnegatively Correlated Increments}
\author{Zachary Selk}
\address{\textit{Zachary Selk} Department of Mathematics, Purdue University, West Lafayette IN 47907, USA.}
\date{\today}
\subjclass[2010]{60G15,60G22, 60K25, 60G55, 60F10}

\maketitle
\begin{abstract}
In this article we prove large deviations principles for high minima of Gaussian processes with nonnegatively correlated increments on arbitrary intervals. Furthermore, we prove large deviations principles for the increments of such processes on intervals $[a,b]$ where $b-a$ is either less than the increment or twice the increment, assuming stationarity of the increments. As a chief example, we consider fractional Brownian motion and fractional Gaussian noise for $H\geq 1/2$. 
\end{abstract}

\section{Introduction}
Let $X(t)$ be a centered continuous Gaussian process with covariance function $R(s,t)=E[X(t)X(s)]$. There has been an abundance of literature on large deviations principles for the maxima of $X(t)$ on an interval $[a,b]$. See e.g. \cite{Berman-1,Berman-2,Berman-3,Berman-4,Berman-5,Berman-6,Berman-7,Kobelkov}. There has however been less work done on large deviations for the minima of $X(t)$ on an interval $[a,b]$. The problem was introduced in \cite{Gennady-High-Level-Sets}, with follow up papers considering the smooth and nonsmooth cases in \cite{Gennady-Asymptotic} and \cite{Gennady-Non-Smooth} respectively. More precisely, setting some high threshold $u>0$, we are interested in the asymptotic behavior of 
\begin{equation*}
   P\left(\min_{t\in [a,b]} X(t)>u\right),
\end{equation*}
as $u\to\infty$.
In \cite{Gennady-High-Level-Sets} it is shown that 
\begin{equation}\label{eq:LDP}
   \lim_{u\to \infty} \frac{1}{u^2}\log \left(P\left(\min_{t\in [a,b]} X(t)>u\right)\right)=-\frac{1}{2\sigma_\ast^2(a,b)},
\end{equation}
where
\begin{equation}\label{eq:optimization}
    \sigma_\ast^2(a,b):=\inf_{\mu\in \mathcal M_1([a,b])}\int_a^b\int_a^b R(s,t)\mu(ds)\mu(dt),
\end{equation}
where $\mathcal M_1$ denotes the space of Borel probability measures on $[a,b]$. This problem has applications in the structure of the high level excursion sets of Gaussian random processes and fields (see \cite{Gennady-Asymptotic,Gennady-High-Level-Sets}). It is also of interest to compute small ball probabilities. 

In addition, the optimization problem $\sigma^2_\ast(a,b)$ is of independent interest and has been generalized in \cite{Minimum-energy-measures}, e.g., by replacing $R(s,t)$ by a general symmetric kernel on some compact set $\mathcal K\subset \mathbb R^d$, and considering different classes of measures instead of $\mathcal M_1$. In \cite{Minimum-energy-measures}, several numerics are given.

In this article, we specialize to a class of processes including fractional Brownian motion and fractional Gaussian noise, which have found wide applications in many fields. For an introduction to fBm, see e.g. \cite{Nourdin-fBm}. FBm finds applications in Queueing theory - see e.g. \cite{Assaf-Queue,Delgado-Queueing,Yamnenko-Queueing}. Fbm also has applications in finance, see e.g. \cite{Oksendal-fBm,Hu-fBm,Gatheral-volatility,Gatheral-volatility-2}. Fractional Gaussian noise has applications in signal processing, see e.g. \cite{Barton-signal}.

Our contributions are as follows. First, we compute $\sigma_\ast^2(a,b)$ in the case where $X$ is a centered continuous Gaussian process with nonnegatively correlated increments on aribitrary interval $[a,b]$. Second, we compute $\sigma_\ast^2(a,b)$ in the case $X$ is the increment of a centered continuous Gaussian process with nonnegatively correlated increments and satisfying some technical conditions on interval $[a,b]$ where $b-a$ is less than the increment. Third, we compute $\sigma_\ast^2(a,b)$ in the case $X$ is the increment of a centered continuous Gaussian process with nonnegatively correlated increments and satisfying some technical conditions on the interval $[a,b]$ where $b-a$ is twice the increment. As a chief example, we compute $\sigma_\ast^2(a,b)$ for fractional Brownian motion and fractional Gaussian noise for $H\geq 1/2$.

\section*{Acknowledgements}
ZS would like to acknowledge helpful conversations with Gennady Samorodnitsky and Harsha Honnappa.

\section{Results}
In order to compute the value $\sigma_\ast^2(a,b)$, we need a technical theorem that allows us to check whether a probability measure is a minimizer in the problem \eqref{eq:optimization}. In this section $X(t)$ will be a centered continuous Gaussian process with covariance function $R(s,t)=E[X(t)X(s)]$. The following result, Theorem 4.3 in \cite{Gennady-High-Level-Sets}, provides a check for the optimal measure in the minimization problem $\sigma_\ast^2(a,b)$.
\begin{theorem}(Theorem 4.3 \cite{Gennady-High-Level-Sets})\label{theorem:check}
Let $X(t)$ be a continuous centered Gaussian process on $[a,b]$ with covariance function $R(s,t)$. Then $\mu^\ast$ is an optimal measure in problem $\sigma_\ast^2(a,b)$ defined in equation \eqref{eq:optimization} if and only if 
\begin{align*}
    \min_{u\in [a,b]}\phi(u)&:= \min_{u\in [a,b]} \int_a^b R(s,u)\mu^\ast(ds)\\
    &=\int_a^b \int_a^b R(s,t) \mu^\ast(ds)\mu^\ast(dt)\\
    &=\phi(t),
\end{align*}
for $\mu^\ast$-a.e. $t\in [a,b]$, where $$\phi(t)=\int_a^b R(s,t)\mu^\ast(ds).$$
\end{theorem}
\begin{remark}
The measure $\mu^\ast$ in Theorem \ref{theorem:check} should be thought of as the distribution of the location of the minima under the limit $u\to \infty$. For an example of a precise statement, see \cite{Gennady-Non-Smooth} Theorem 4.1.
\end{remark}

This section will be broken up into two subsections. The first subsection will handle the case when $X$ is a centered continuous Gaussian process with nonnegatively correlated increments. The second subsection will handle the case where $X$ is the increment of a centered continuous Gaussian process with nonnegatively correlated stationary increments.

\subsection{Processes with Nonnegatively Correlated Increments}
In this subsection, we consider continuous centered Gaussian processes with nonnegatively correlated increments. 
\begin{proposition}\label{proposition:non-increment}
Let $X(t)$ be a continuous centered Gaussian process on the interval $[a,b]$ with covariance function $R(s,t)$ and $X(0)=0$ a.s. Assume that $X$ has nonnegatively correlated increments. That is, for all quadruples $(s_1,t_1,s_2,t_2)$ with $a\leq s_1\leq t_1\leq s_2\leq t_2\leq b$ we have that
$$E\left[(X(t_1)-X(s_1))(X(t_2)-X(s_2))\right]\geq 0.$$
Then, the optimal measure $\mu^\ast=\delta_a$.
\end{proposition}
\begin{proof}
By Theorem \ref{theorem:check}, we must show the following two things
\begin{equation*}
    \min_{u\in [a,b]}\phi(u):= \min_{u\in [a,b]} \int_a^b R(s,u)\mu^\ast(ds)=\int_a^b \int_a^b R(s,t) \mu^\ast(ds)\mu^\ast(dt)
\end{equation*}
and 
\begin{equation*}
    \min_{u\in [a,b]}\phi(u)=\phi(t),
\end{equation*}
for $t=\{a\}$. To this aim, we compute that 
\begin{equation*}
    \int_a^b R(s,t) \delta_a(ds)=R(a,t).
\end{equation*}
Taking a minimum achieves that
\begin{align*}
    \min_{t\in [a,b]}R(a,t)&=\min_{t\in [a,b]}E[X(t)X(a)]\\
    &=\min_{t\in [a,b]}E[(X(t)-X(a)+X(a))X(a)]\\
    &=\min_{t\in [a,b]}E[(X(t)-X(a))(X(a)-X(0))]+E[X^2(a)].
\end{align*}
By the assumption that the increments are nonnegatively correlated, we conclude that this minima is achieved at $t=a$. 
\end{proof}
There is a partial converse to Proposition \ref{proposition:non-increment}. 
\begin{proposition}\label{proposition:converse-1}
Let $X(t)$ be a continuous centered Gaussian process with covariance function $R(s,t)$ and $X(0)=0$ a.s. Suppose that on any interval $[a,b]$ the optimal measure $\mu^\ast=\delta_a$. Then $X$ satisfies 
\begin{equation*}
    E[(X(t)-X(a))(X(a)-X(0))]\geq 0,
\end{equation*}
for all $0\leq a\leq t$.
\end{proposition}
\begin{proof}
By Theorem \ref{theorem:check}, we have that for any interval $[a,b]$
\begin{align*}
    \min_{t\in [a,b]}R(a,t)&=\min_{t\in [a,b]}E[X(t)X(a)]\\
    &=\min_{t\in [a,b]} E[(X(t)-X(a)+X(a))X(a)]\\
    &=\min_{t\in [a,b]}E[(X(t)-X(a))(X(a)-X(0))]+E[X^2(a)]\\
    &=E[X^2(a)]+\min_{t\in [a,b]}E[(X(t)-X(a))(X(a)-X(0))]\\
    &=E[X^2(a)].
\end{align*}
Therefore, we have that for $0\leq a\leq t$ that 
\begin{equation*}
    E[(X(t)-X(a))(X(a)-X(0))]\geq 0.
\end{equation*}
\end{proof}

\subsection{Increment of Processes with Nonnegatively Correlated Increments}
In this subsection we find the optimal measure for the increment process, $X_Y(t)=Y(t+h)-Y(t)$ of a continuous centered Gaussian process $Y(t)$ with stationary increments whose increment function $f_Y^h(t)=E[Y^2(t+h)]-E[Y^2(t)]$ satisfies some growth conditions which are easily numerically verified for examples chiefly including fractional Gaussian noise for $H\geq 1/2$. 

First, we need a technical lemma which will let us handle the case where $[a,b]$ is an interval such that $b-a\leq h$. 

\begin{lemma}\label{lemma:function}
Let $\rho \in C^1((0,h),\mathbb R)$ be a continuously differentiable function such that $\rho'>0$ and $\rho'$ is decreasing on the interval $(0,h)$ for some $h>0$. Then for $b\in (0,h)$ the function $\phi(t):=\rho(t)+\rho(b-t)$ has a unique extrema on $(0,b)$ at $t=b/2$, which is a maxima.
\end{lemma}
\begin{proof}
As $\rho$ and thus $\phi$ is differentiable, we have that 
$$\phi'(t)=\rho'(t)-\rho'(b-t).$$
Therefore, $\phi$ has a local extrema if and only if
$$\rho'(t)=\rho'(b-t).$$
If $t\neq b/2$ satisfies the above, this contradicts the monotonicity of $\rho'$. Therefore $t=b/2$ is the only solution. By concavity, it is a maxima.
\end{proof}
The following lemma decomposes the covariance function of an increment process into terms involving the increment function $f_Y^h$. This will ease analysis and make useful Lemma \ref{lemma:function} in the computations to come. 
\begin{lemma}\label{lemma:covariance}
Let $Y(t):[0,\infty)\to \mathbb R$ be a continuous centered Gaussian process which has stationary increments with covariance function $R_Y(s,t)$ and variance function $V_Y(t)=R_Y(t,t)$. We also extend the variance process to all of $\mathbb R$ by imposing that $V_Y(t)=V_Y(-t)$ for all $t\in \mathbb R$. Consider the increment process $X_Y(t)=Y(t+h)-Y(t)$ for fixed $h>0$. Then $X_Y$ is a continuous centered Gaussian process with covariance function 
\begin{equation}
    \Gamma(t-s)=R(s,t)=\frac12(f_Y^h(t-s)+f_Y^h(s-t)),
\end{equation}
where we have denoted $f_Y^h(t):=V_Y(t+h)-V_Y(t).$
\end{lemma}
\begin{proof}
$X_Y$ is clearly a continuous centered Gaussian process, so we only need to compute the covariance. 
\begin{align*}
    R(s,t)&=E[X_Y(s)X_Y(t)]\\
    &=E[(Y(s+h)-Y(s))(Y(t+h)-Y(t))]\\
    &=R_Y(s+h,t+h)-R_Y(s,t+h)-R_Y(t,s+h)+R_Y(s,t).
\end{align*}
By stationarity of increments, we have that 
$$R_Y(u,v)=\frac{1}{2}\left(V_Y(u)+V_Y(v)-V_Y(u-v)\right).$$
Using this expression for $R_Y$ implies that
\begin{align*}
    2R(s,t)&=V_Y(s+h)+V_Y(t+h)-V_Y(s-t)-V_Y(s)-V_Y(t+h)+V_Y(s-t-h)\\
    &~-V_Y(t)-V_Y(s+h)+V_Y(t-s-h)+V_Y(s)+V_Y(t)-V_Y(t-s)\\
    &=-V_Y(s-t)+V_Y(s-t-h)+V_Y(t-s-h)-V_Y(t-s).
\end{align*}
Recalling that we imposed $V_Y(u)=V_Y(-u)$ for $u\in \mathbb R$ concludes that
\begin{align*}
    2R(s,t)&=-V_Y(t-s)+V_Y(h+t-s)+V_Y(t-s-h)-V_Y(t-s)\\
    &=-V_Y(t-s)+V_Y(h+t-s)+V_Y(-t+s+h)-V_Y(t-s)\\
    &=f_Y^h(t-s)+f_Y^h(s-t).
\end{align*}
\end{proof}
In light of the above decomposition of $\Gamma$ into increment functions $f_Y^h$, the following lemma describes the relevant behavior of the increment function of processes in interest, and motivates our key Assumption \ref{assumption:for-first case}.
\begin{lemma}
Let $Y(t):[0,\infty)\to \mathbb R$ be a centered continuous Gaussian process with stationary nonnegatively correlated increments and covariance function $R_Y(s,t)$ and variance function $V_Y(t)=R_Y(t,t)$. Assume that $V_Y$ is differentiable on $(0,\infty)$. Also assume that $Y(0)=0$ a.s. Then the increment function $f_Y^h(t):\mathbb R\to \mathbb R$ defined by $f_Y^h(t)=V_Y(t+h)-V_Y(t)$, where we have extended $V_Y(t)=V_Y(-t)$ for $t\in (-\infty,0)$ is increasing on $\mathbb R$. 
\end{lemma}
\begin{proof}
Let $t_2>t_1>0$. Using stationary increments of $Y$ and thus the relation $V_Y(t+h)-V_Y(t)=2R_Y(t+h,h)-V_Y(h)$, we arrive at
\begin{align*}
f_Y^h(t_2)-f_Y^h(t_1)&=V_Y(t_2+h)-V_Y(t_2)-V_Y(t_1+h)+V_Y(t_1)\\
&=2R_Y(t_2+h,h)-V_Y(h)-2R_Y(t_1+h,h)+V_Y(h)\\
&=2 E[(Y(t_2+h)-Y(t_1+h))Y(h)]\\
&=2 E[(Y(t_2+h)-Y(t_1+h))(Y(h)-Y(0)]\\
&\geq 0,
\end{align*}
where the last inequality is because $Y$ has nonnegatively correlated increments.\\
Now, let $-h<t<0$. Then 
$$f_Y^h(t)=V_Y(h+t)-V_Y(-t),$$ 
where $h+t>0$ and $-t>0$. Therefore by assumption, we may differentiate $f_Y^h$ to get that
$$(f_Y^h)'(t)=V_Y'(h+t)+V_Y'(-t)>0+0=0.$$
By Proposition \ref{proposition:non-increment}, we know that $V_Y'(u)>0$ for $u>0$. \\
Finally, let $t<-h$. Then 
$$f_Y^h(t)=V_Y(-t-h)-V_Y(-t)=-f_Y^h(-t).$$
Again, differentiation shows that 
$$(f_Y^h)'(t)>0.$$
\end{proof}
In light of our decomposition of $\Gamma$ Lemma \ref{lemma:covariance} and the above lemma, the relevant assumptions are on the increment function $f_Y^h$. Therefore we state our assumptions on the process $Y$, its increment $Y(t+h)-Y(t)$ and its increment function $f_Y^h$.
\begin{assumption}\label{assumption:for-first case}
Let $Y(t):[0,\infty)\to \mathbb R$ be a centered continuous Gaussian process with stationary increments and covariance function $R_Y(s,t)$ and variance function $V_Y(t)$, where we extend $V_Y(t)=V_Y(-t)$ for $t\in (-\infty,0)$. Let $X_Y(t)=Y(t+h)-Y(t)$ denote the increment process for fixed $h>0$ with covariance function $\Gamma(t-s)=R(s,t)$. We assume that the increment function $f_Y^h(t)=V_Y(t+h)-V_Y(t)$ is $C^2((0,\infty),\mathbb R)$. We assume that $(f_Y^h)'(t)>0$ for all $t\in \mathbb R$. We also assume that for all $b\in [0,h]$ $(f_Y^h)''(t)+(f_Y^h)''(t-b)<0$, for all $t\in [0,\infty)/\{b\}$.
\end{assumption}
\begin{remark}
The last line in Assumption \ref{assumption:for-first case} is easily verified numerically. It is difficult in general to prove from the nonnegatively correlated increments of $Y$. It is true for fractional Gaussian noise. 
\end{remark}
With the above assumption and lemmas, we are able to state and prove first our result for processes satisfying Assumption \ref{assumption:for-first case} on the interval $[a,b]$ with $b-a\leq h$. Later, we will state our result on the interval $[a,b]$ with $b-a=2h$.
\begin{theorem}
Let $Y$ be a function satisfying the assumption \ref{assumption:for-first case} with increment process $X_Y(t)$. Consider the interval $[a,b]$ with $b-a\leq h$. Then the optimal measure associated to the optimization problem $\sigma_\ast^2(a,b)$ defined in \eqref{eq:optimization} is $$\mu^\ast=\frac{1}{2}\left(\delta_a+\delta_b\right).$$
\end{theorem}
\begin{proof}
By Assumption \ref{assumption:for-first case}, the process $X_Y$ is stationary. Therefore without loss of generality we may assume that $a=0$ and $b\leq h$. Recall that by Theorem \ref{theorem:check} we only have to verify that
\begin{equation*}
    \min_{u\in [0,b]} \int_0^b R(s,u) \mu^\ast(ds)=\int_0^b\int_0^b R(s,t) \mu^\ast(ds) \mu^\ast(dt),
\end{equation*}
and 
\begin{equation*}
    \min_{u\in [0,b]} \int_0^b R(s,u) \mu^\ast(ds)=\int_0^b R(s,t) \mu^\ast(ds)
\end{equation*}
for $t\in \{0,b\}$. To verify both properties we note that 
\begin{equation*}
    \phi(u):=\int_0^b R(s,u) \mu^\ast(ds)=\frac{1}{2}\left(\Gamma(u)+\Gamma(b-u)\right).
\end{equation*}
Recalling from Lemma \ref{lemma:covariance} that $\Gamma(u)=f_Y^h(u)+f_Y^h(-u)$, we may rewrite $\phi$ as
\begin{equation*}
    \phi(u)=\frac{1}{4}\left(f_Y^h(u)+f_Y^h(-u)+f_Y^h(b-u)+f_Y^h(u-b)\right).
\end{equation*}
Using the notation of Lemma \ref{lemma:function}, we write
\begin{equation*}
    \rho(t)=\frac{1}{4}(f_Y^h(u)+f_Y^h(u-b)).
\end{equation*}
By assumption, we have that $\rho'(t)>0$ and $\rho''(t)<0$ for $t\in [0,b]$. Thus Lemma \ref{lemma:function} says there is a unique extreme point of $\phi(t)$ on $[0,b]$ at $t=b/2$. Therefore 
\begin{equation*}
    \min_{u\in [0,b]} \phi(u)=\min \{\phi(0),\phi(b)\}=\phi(0)=\phi(b),
\end{equation*}
which verifies both properties. 
\end{proof}
In order to handle the case where $b-a=2h$, we need one final assumption.

\begin{assumption}\label{assumption:for-second-case}
Let $Y$ be a process satisfying Assumption \ref{assumption:for-first case} with increment process $X_Y$. We also assume that 
\begin{equation}
    C^\ast:=1+\frac{\Gamma(h)-\Gamma(2h)}{\Gamma(h)-\Gamma(0)}>0,
\end{equation}
and 
\begin{equation}
    \gamma(t)=\Gamma(t)+C^\ast\Gamma(h-t)+\Gamma(2h-t)
\end{equation}
has at most one critical point, a maxima on $(0,h)$.
\end{assumption}

Now we can state and prove our result for processes satisfying Assumption \ref{assumption:for-first case} on the interval $[a,b]$ where $b-a=2h$. 

\begin{theorem}\label{theorem:second-case}
Let $Y$ be a process satisfying Assumptions \ref{assumption:for-first case} and \ref{assumption:for-second-case} with increment process $X_Y$. Consider the interval $[a,b]$ with $b-a= 2h$. Then the optimal measure associated to the optimization problem $\sigma_\ast^2(a,b)$ defined in \eqref{eq:optimization} for $X_Y$ is $$\mu^\ast =\frac{1}{2+C^\ast}(\delta_a+C^\ast \delta_{a+h}+\delta_b).$$ \end{theorem}
\begin{proof}
By Assumption \ref{assumption:for-first case}, $X_Y$ is a stationary process. Therefore without loss of generality we may work on the interval $[0,2h]$. Recall that by Theorem \ref{theorem:check} we only have to verify that
\begin{equation*}
    \min_{u\in [0,2h]} \int_0^b R(s,u) \mu^\ast(ds)=\int_0^b\int_0^b R(s,t) \mu^\ast(ds) \mu^\ast(dt),
\end{equation*}
and 
\begin{equation*}
    \min_{u\in [0,2h]} \int_0^b R(s,u) \mu^\ast(ds)=\int_0^b R(s,t) \mu^\ast(ds)
\end{equation*}
for $t\in \{0,h,2h\}$. We use the definition of $\mu^\ast$ to get that 
\begin{equation*}
    \phi(t):=\int_0^b R(s,u) \mu^\ast(ds)=\frac{1}{2+C^\ast}\left(\Gamma(t)+C^\ast \Gamma(h-t)+\Gamma(2h-t)\right).
\end{equation*}
By assumption, we can verify that $\phi(t)$ has no minima on $(0,h)$. Thus by symmetry $\phi$ can have at most one extrema on the interval $(h,2h)$ which would also be a maxima, as well. Therefore there are no minima in the interval $(h,2h)$ either. Then we just need to check that 
\begin{equation*}
    \phi(0)=\phi(h)=\phi(2h).
\end{equation*}
Again, as $\phi(t)=\phi(2h-t)$ we only need to compute $\phi(0)$ and $\phi(h)$. They are
\begin{equation*}
    \phi(0)=\frac{1}{2+C^\ast}\left(\Gamma(0)+C^\ast \Gamma(h)+\Gamma(2h)\right)
\end{equation*}
and 
\begin{equation*}
    \phi(h)=\frac{1}{2+C^\ast}\left(\Gamma(h)+C^\ast \Gamma(0)+\Gamma(h)\right).
\end{equation*}
By definition of $C^\ast$, we have that 
\begin{equation*}
    \phi(0)-\phi(h)=\frac{1}{2+C^\ast}\left(\Gamma(0)+\Gamma(2h)-2\Gamma(h)+C^\ast(\Gamma(h)-\Gamma(0)\right)=0
\end{equation*}
\end{proof}

\begin{remark}
The assumption that $\gamma(t)$ has no minima on $(0,h)$ in Theorem \ref{theorem:second-case} can be verified numerically easily. It is however difficult to verify in generality. See the figures at the end of the article for numerical verification. 
\end{remark}

\section{Fractional Brownian Motion and Fractional Gaussian Noise}\label{section:Examples}
In this section we give examples of Gaussian processes that satisfy Assumptions \ref{assumption:for-first case} and \ref{assumption:for-second-case}.
\begin{example}
An example of a continuous centered Gaussian process with nonnegatively correlated increments is fractional Brownian motion, $B_H$ with Hurst index $H\geq 1/2$. Fractional Brownian motion is a process whose covariance function is
\begin{equation*}
    R(s,t)=\frac{1}{2}\left(t^{2H}+s^{2H}-|t-s|^{2H}\right) ,
\end{equation*}
and variance function is $V(t)=|t|^{2H}$. When $H=1/2$, one recovers standard Brownian motion. It is well known that fractional Brownian motion has stationary nonnegatively correlated increments for $H\geq 1/2$ - see e.g. \cite{Nourdin-fBm}. Therefore Proposition \ref{proposition:non-increment} applies and the optimal measure for $B_H$ on the interval $[a,b]$ is $\mu^\ast=\delta_a$. Therefore the rate function for the optimization problem $\sigma_\ast^2(a,b)$ is 
\begin{equation*}
   \lim_{u\to \infty} \frac{1}{u^2}\log \left(P\left(\min_{t\in [a,b]} B_H(t)>u\right)\right)= -\frac{1}{2a^{2H}}.
\end{equation*}
\end{example}
\begin{example}
The increment of fractional Gaussian Brownian motion is called fractional Gaussian noise. The covariance function for fractional Gaussian noise is 
\begin{equation*}
    \Gamma(t-s)=\frac{1}{2}\left( |t-s-h|^{2H}-2|t-s|^{2H}+|t-s+h|^{2H} \right),
\end{equation*}
and increment function
\begin{equation*}
    f_Y^h(t)=|t+h|^{2H}-|t|^{2H}, 
\end{equation*}
whose first derivative is
\begin{equation*}
    (f_Y^h)'(t)=2H\left(\frac{t+h}{|t+h|^{2-2H}}-\frac{t}{|t|^{2-2H}}\right)
\end{equation*}
and second derivative is
\begin{equation*}
    (f_Y^h)''(t)=2H(2H-1)\left(\frac{(t+h)^{2}}{|t+h|^{4-2H}}-\frac{t^2}{|t|^{4-2H}}\right)
\end{equation*}
Assumptions \ref{assumption:for-first case} and \ref{assumption:for-second-case} can be verified numerically in the case of fractional Gaussian noise, as the following figures will illustrate. In this case, we have that
\begin{equation*}
    \sigma_\ast^2(a,b)=\frac{1}{C^\ast+2}\left(\Gamma(0)+C^\ast \Gamma(h)+\Gamma(2h)\right),
\end{equation*}
and
\begin{equation*}
     \lim_{u\to \infty} \frac{1}{u^2}\log \left(P\left(\min_{t\in [a,b]} B_H(t+h)-B_H(t)>u\right)\right)= -\frac{C^\ast+2}{2(\Gamma(0)+C^\ast \Gamma(h)+\Gamma(2h))}.
\end{equation*}
\end{example}
\section{Further Questions}
We conclude the paper by asking some questions for possible future directions of research. 
\begin{question}
Can Proposition \ref{proposition:converse-1} be strengthened by adding additional constraints on $X$, for example if $X$ has stationary increments?
\end{question}
\begin{question}
Can the assumption that $\gamma(t)$ has no minima on $(0,h)$ in Theorem \ref{theorem:second-case} be verified in generality under Assumption \ref{assumption:for-first case}?
\end{question}
\begin{question}
Does the last line of Assumption \ref{assumption:for-first case} hold for general Gaussian processes with stationary nonnegatively correlated increments?
\end{question}
\begin{question}
What is the optimal measure for the increment of processes $Y$ satisfying Assumption \ref{assumption:for-first case} on intervals $(0,nh)$ for $n\in \{3,4,...\}$?
\end{question}
\begin{question}
What is the optimal measure for the increment of processes $Y$ satisfying Assumption \ref{assumption:for-first case} on intervals $(0,b)$ for $b\neq nh$ and $b>h$?
\end{question}
\FloatBarrier
\begin{figure}[p]
    \centering
    \includegraphics[width=90mm]{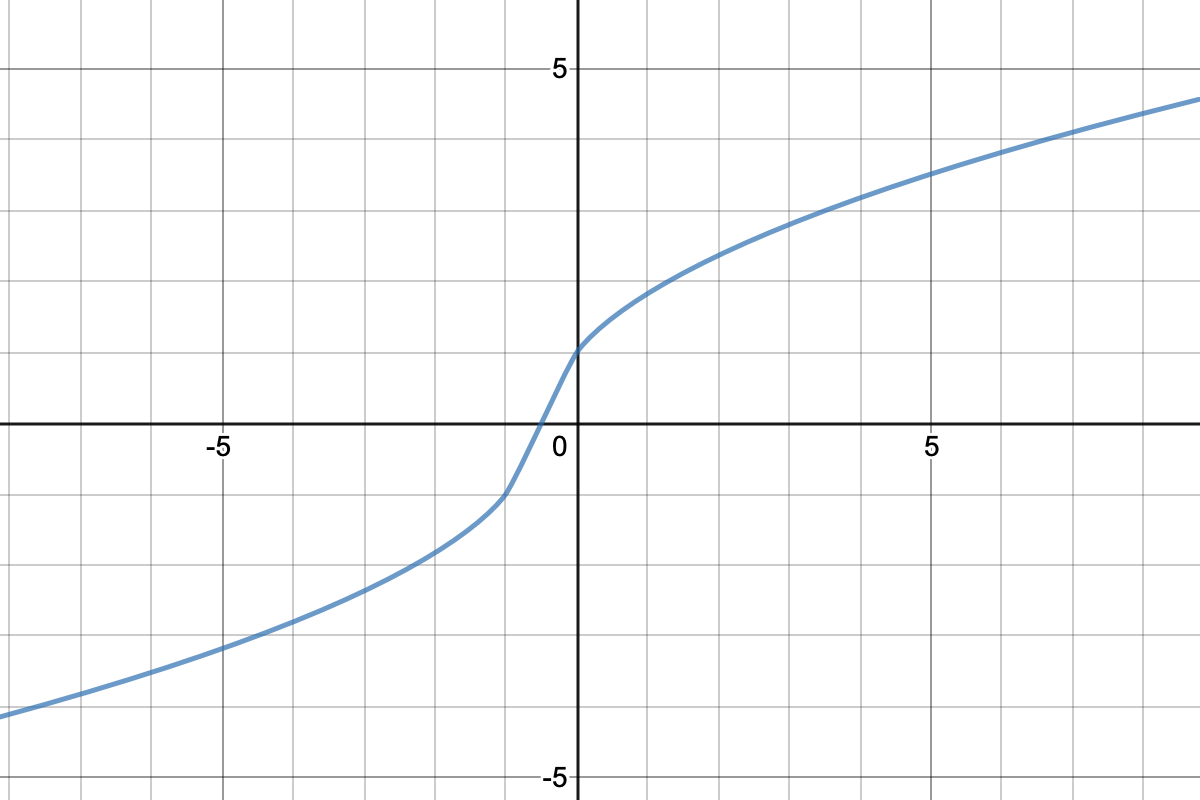}
    \caption{The increment process $f_Y^h$ for fBm with $2H=1.5$, $h=1$}
\end{figure}
\begin{figure}[p]
    \centering
    \includegraphics[width=90mm]{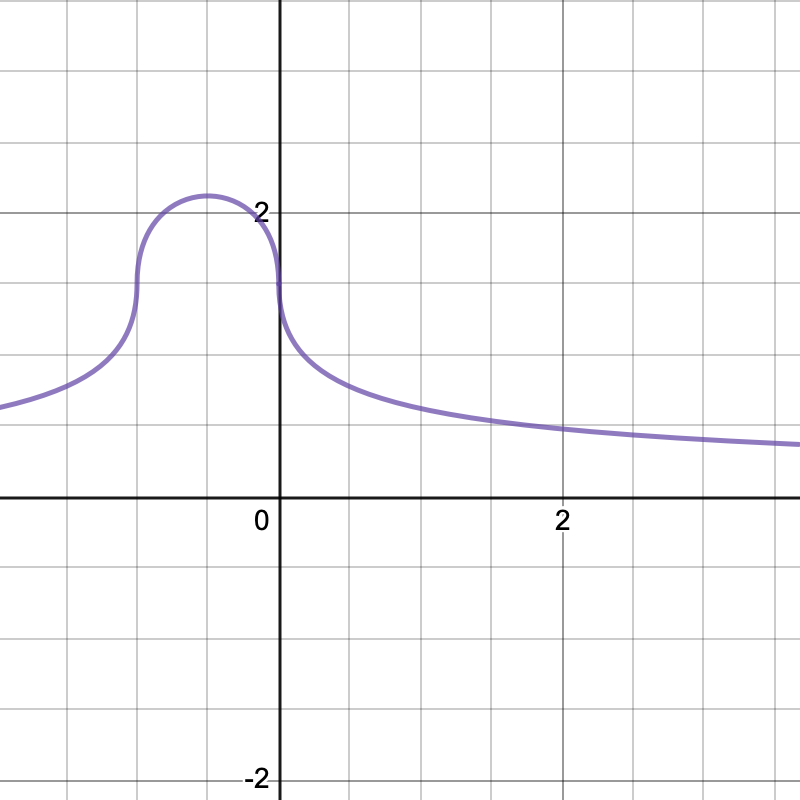}
    \caption{The first derivative of the increment process $f_Y^h$ for fBm with $2H=1.5$, $h=1$}
\end{figure}
\begin{figure}[p]
    \centering
    \includegraphics[width=90mm]{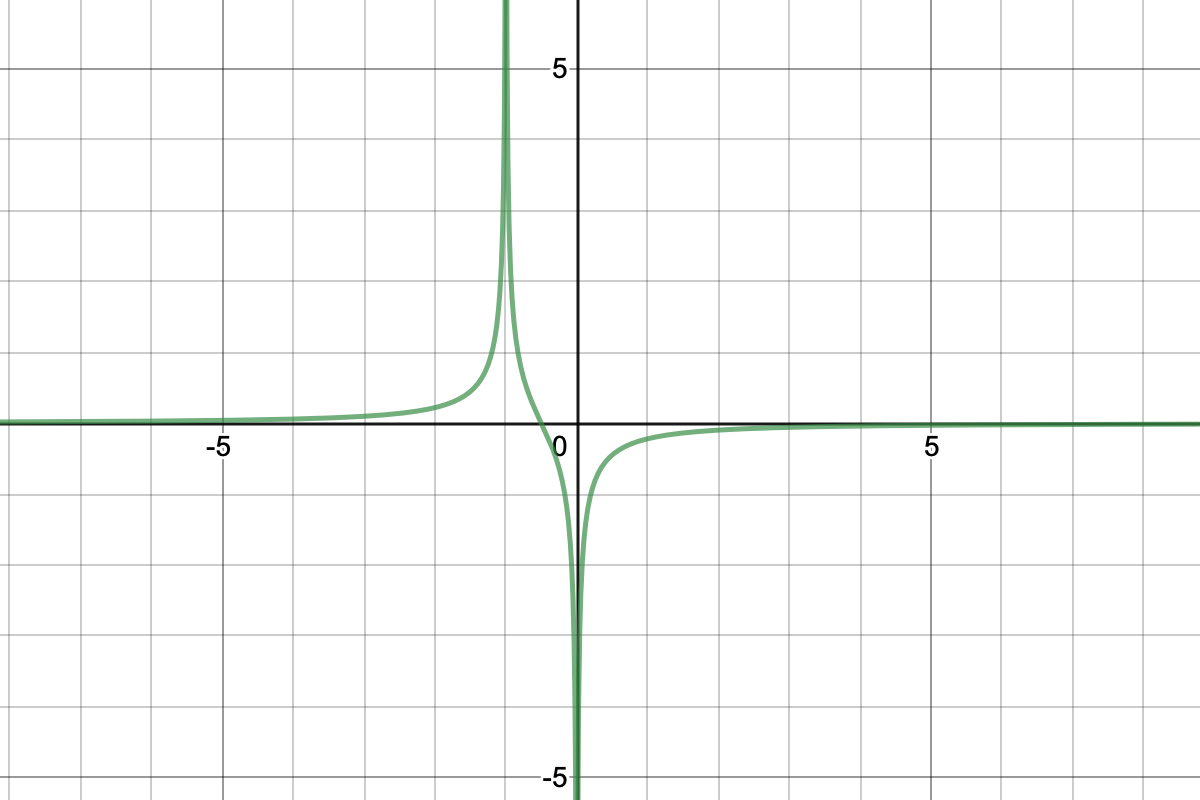}
    \caption{The second derivative of the increment process $f_Y^h$ for fBm with $2H=1.5$, $h=1$}
\end{figure}
\begin{figure}[p]
    \centering
    \includegraphics[width=90mm]{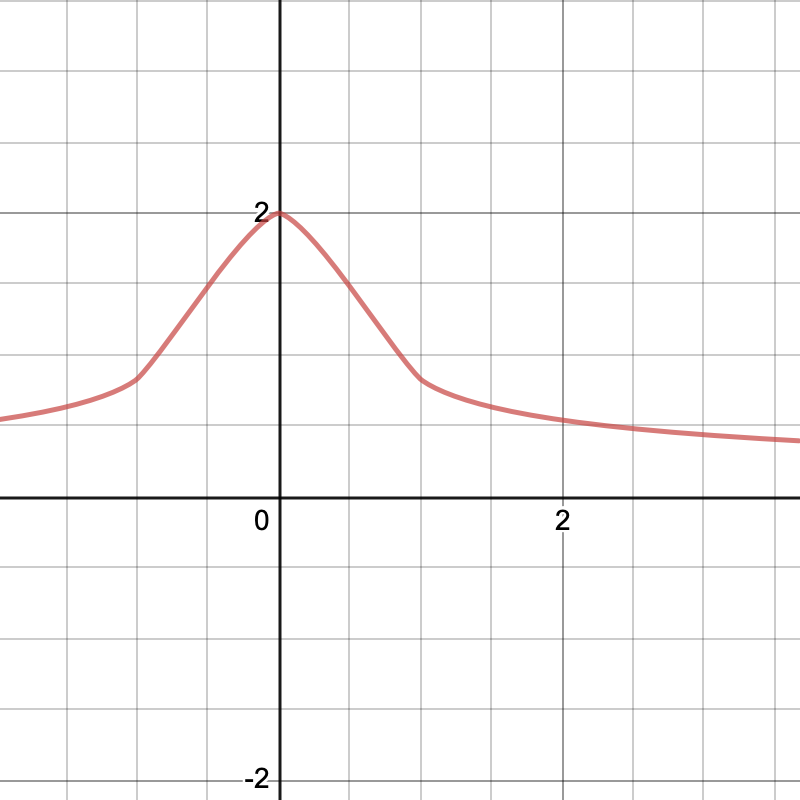}
    \caption{The covariance function $\Gamma$ for fBm with $2H=1.5$, $h=1$}
\end{figure}
\begin{figure}[p]
    \centering
    \includegraphics[width=90mm]{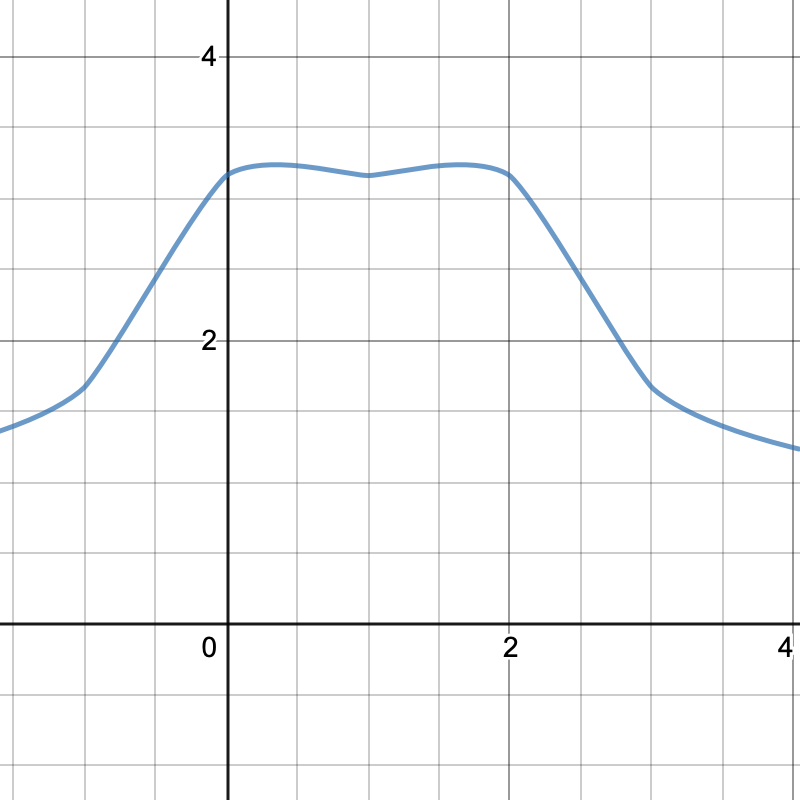}
    \caption{The function $\gamma$ for fBm with $2H=1.5$, $h=1$}
\end{figure}
\begin{figure}[p]
    \centering
    \includegraphics[width=90mm]{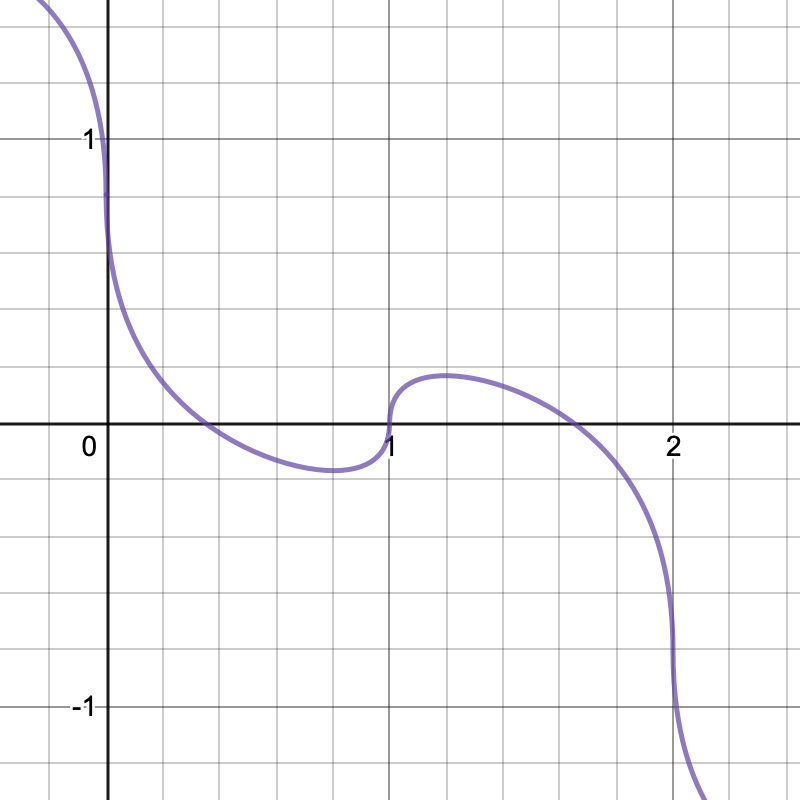}
    \caption{The derivative of the function $\gamma$ for fBm with $2H=1.5$, $h=1$. Notice how there is one critical point on $(0,h)=(0,1)$.}
\end{figure}

\FloatBarrier

\newpage
\bibliographystyle{plain}
\bibliography{reference.bib}

\begin{thebibliography}{10}

\bibitem{Gennady-High-Level-Sets}
Robert~J. Adler, Elina Moldavskaya, and Gennady Samorodnitsky.
\newblock On the existence of paths between points in high level excursion sets
  of {G}aussian random fields.
\newblock {\em Ann. Probab.}, 42(3):1020--1053, 2014.

\bibitem{Barton-signal}
Richard~J. Barton and H.~Vincent Poor.
\newblock Signal detection in fractional {G}aussian noise.
\newblock {\em IEEE Trans. Inform. Theory}, 34(5, part 1):943--959, 1988.

\bibitem{Gatheral-volatility-2}
Christian Bayer, Peter Friz, and Jim Gatheral.
\newblock Pricing under rough volatility.
\newblock {\em Quant. Finance}, 16(6):887--904, 2016.

\bibitem{Berman-7}
Simeon~M. Berman.
\newblock Maximum and high level excursion of a {G}aussian process with
  stationary increments.
\newblock {\em Ann. Math. Statist.}, 43:1247--1266, 1972.

\bibitem{Berman-5}
Simeon~M. Berman.
\newblock An asymptotic bound for the tail of the distribution of the maximum
  of a {G}aussian process.
\newblock {\em Ann. Inst. H. Poincar\'{e} Probab. Statist.}, 21(1):47--57,
  1985.

\bibitem{Berman-6}
Simeon~M. Berman.
\newblock An asymptotic formula for the distribution of the maximum of a
  {G}aussian process with stationary increments.
\newblock {\em J. Appl. Probab.}, 22(2):454--460, 1985.

\bibitem{Berman-3}
Simeon~M. Berman.
\newblock The maximum of a {G}aussian process with nonconstant variance.
\newblock {\em Ann. Inst. H. Poincar\'{e} Probab. Statist.}, 21(4):383--391,
  1985.

\bibitem{Berman-4}
Simeon~M. Berman.
\newblock Sojourns above a high level for a {G}aussian process with a point of
  maximum variance.
\newblock {\em Comm. Pure Appl. Math.}, 38(5):519--528, 1985.

\bibitem{Berman-2}
Simeon~M. Berman.
\newblock Extreme sojourns of a {G}aussian process with a point of maximum
  variance.
\newblock {\em Probab. Theory Related Fields}, 74(1):113--124, 1987.

\bibitem{Berman-1}
Simeon~M. Berman and Norio K\^{o}no.
\newblock The maximum of a {G}aussian process with nonconstant variance: a
  sharp bound for the distribution tail.
\newblock {\em Ann. Probab.}, 17(2):632--650, 1989.

\bibitem{Gennady-Asymptotic}
Arijit Chakrabarty and Gennady Samorodnitsky.
\newblock Asymptotic behaviour of high {G}aussian minima.
\newblock {\em Stochastic Process. Appl.}, 128(7):2297--2324, 2018.

\bibitem{Delgado-Queueing}
Rosario Delgado.
\newblock On the reflected fractional {B}rownian motion process on the positive
  orthant: asymptotics for a maximum with application to queueing networks.
\newblock {\em Stoch. Models}, 26(2):272--294, 2010.

\bibitem{Gatheral-volatility}
Jim Gatheral, Thibault Jaisson, and Mathieu Rosenbaum.
\newblock Volatility is rough.
\newblock {\em Quant. Finance}, 18(6):933--949, 2018.

\bibitem{Hu-fBm}
Yaozhong Hu, Bernt \O~ksendal, and Donna~Mary Salopek.
\newblock Weighted local time for fractional {B}rownian motion and applications
  to finance.
\newblock {\em Stoch. Anal. Appl.}, 23(1):15--30, 2005.

\bibitem{Kobelkov}
S.~G. Kobel'~kov.
\newblock The probability of exceeding a high level for the trajectories of a
  {G}aussian process with dispersions that attain the maximum on discrete sets.
\newblock {\em Fundam. Prikl. Mat.}, 22(3):83--90, 2018.

\bibitem{Nourdin-fBm}
Ivan Nourdin.
\newblock {\em Selected aspects of fractional {B}rownian motion}, volume~4 of
  {\em Bocconi \& Springer Series}.
\newblock Springer, Milan; Bocconi University Press, Milan, 2012.

\bibitem{Oksendal-fBm}
Bernt \O~ksendal.
\newblock Fractional {B}rownian motion in finance.
\newblock In {\em Stochastic economic dynamics}, pages 11--56. Cph. Bus. Sch.
  Press, Frederiksberg, 2007.

\bibitem{Minimum-energy-measures}
Luc Pronzato and Anatoly Zhigljavsky.
\newblock Minimum-energy measures for singular kernels.
\newblock {\em J. Comput. Appl. Math.}, 382:113089, 16, 2021.

\bibitem{Gennady-Non-Smooth}
Zhixin Wu, Arijit Chakrabarty, and Gennady Samorodnitsky.
\newblock High minima of non-smooth {G}aussian processes.
\newblock {\em Electron. Commun. Probab.}, 24:Paper No. 53, 12, 2019.

\bibitem{Yamnenko-Queueing}
R.~Yamnenko, Yu. Kozachenko, and D.~Bushmitch.
\newblock Generalized sub-{G}aussian fractional {B}rownian motion queueing
  model.
\newblock {\em Queueing Syst.}, 77(1):75--96, 2014.

\bibitem{Assaf-Queue}
Assaf~J. Zeevi and Peter~W. Glynn.
\newblock On the maximum workload of a queue fed by fractional {B}rownian
  motion.
\newblock {\em Ann. Appl. Probab.}, 10(4):1084--1099, 2000.

\end{thebibliography}

\end{document}